\documentclass[12pt]{article}
\usepackage{amsmath,amsfonts,amssymb,amsthm}
\usepackage{bbm,mathrsfs,bm,mathtools}
\title{Radial measures of pseudo-cones}
\author{Rolf Schneider}
\date{}
\newcommand{\Sn}{{\mathbb S}^{n-1}}

\newcommand{\R}{{\mathbb R}}

\newcommand{\K}{{\mathcal K}}

\newcommand{\N}{{\mathbb N}}

\newcommand{\cH}{\mathcal{H}}
\newcommand{\Ha}{\mathcal{H}}

\newcommand{\D}{{\rm d}}

  \newcommand{\fed}{\,\rule{.1mm}{.20cm}\rule{.20cm}{.1mm}\,}

\newtheorem{theorem}{Theorem}
\newtheorem{lemma}{Lemma}

\begin{document}
\maketitle

\begin{abstract}
{We consider $C$-pseudo-cones, that is, closed convex sets $K\subset\R^n$ with $o\notin K\subset C$, for which $C$ is the recession cone. Here $C$ is a given closed convex cone in $\R^n$, pointed and with nonempty interior. We define a class of measures for such pseudo-cones and show how they can be interpreted as derivative measures. For a subclass of these measures, namely for dual curvature measures with negative indices, we solve a Minkowski type existence problem.}\\[2mm]
{\em Keywords:} pseudo-cone, radial measure, derivative measure, dual curvature measure, Minkowski type theorem   \\[1mm]
2020 Mathematics Subject Classification: 52A20 
\end{abstract}

\section{Introduction}\label{sec1}

The Minkowski problem for convex bodies has in recent years been the template for several generalizations and analogues. Let us first recall the classical issue. For a convex body $K$ (a compact convex set with interior points) in $\R^n$, the surface area measure $S_K$ is defined as the push-forward of the $(n-1)$-dimensional Hausdorff measure $\cH^{n-1}$, restricted to the boundary ${\rm bd}\,K$ of $K$, under the ($\Ha^{n-1}$-almost everywhere on ${\rm bd}\,K$ defined) Gauss map of $K$. This can be written as
\begin{equation}\label{1.1} 
S_K= \nu_K{_\sharp}(\cH^{n-1}\fed{{\rm bd}\,K}),
\end{equation}
where the Gauss map $\nu_K$ of $K$ is defined by letting $\nu_K(x)$ be the outer unit normal vector of $K$ at $x\in {\rm bd }\,K$ if it is uniquely defined. Thus, $S_K$ is a finite Borel measure on the unit sphere $\Sn$. Equivalently, it can be defined as a derivative measure. If $V_n$ denotes the volume in $\R^n$ and $L$ is an arbitrary convex body with support function $h_L$, then
\begin{equation}\label{1.2}
\frac{\D}{\D t} V_n(K+tL)\Big|_{t=0^+} = \int_{\Sn} h_L(u)\,S_K(\D u),
\end{equation}
where $K+tL$ denotes the Minkowski sum of $K$ and $tL$. The Minkowski problem now asks: What are the necessary and sufficient conditions on a measure $\mu$ on $\Sn$ in order that there is a convex body $K$ with $S_K=\mu$? And if it exists, in how far is it unique? In this classical case, the answers are well known (see, e.g., \cite[Chap. 8]{Sch14}).

We note that in (\ref{1.2}) the derivative is one-sided, since the definition of the Minkowski addition requires that $t\ge 0$. This can be remedied by the introduction of the Wulff shape $[h]$ of a positive function $h$ on the unit sphere and a corresponding lemma of Aleksandrov; we refer to \cite[Sect. 7.5]{Sch14}. We then have, if $o \in{\rm int}\,K$, 
\begin{equation}\label{1.2a}
\frac{\D}{\D t} V_n([h_K+th_L])\Big|_{t=0} = \int_{\Sn} h_L(u)\,S_K(\D u).
\end{equation}

Equation (\ref{1.1}) is of the type
\begin{equation}\label{1.3} 
\mu= \varphi_K{_\sharp} m_K,
\end{equation}
where $m_k$ is some standard measure, possibly with a density, defined by $K$, and $\varphi_K$ is some geometric measurable mapping, also defined by $K$, from the domain of $m_K$ to the unit sphere. The equation (\ref{1.2a}) can be written in the form
\begin{equation}\label{1.4}
\frac{\D}{\D t} \Phi(K_t)\Big|_{t=0} = \int_{\Sn} g(u)\,\mu(\D u),
\end{equation}
where $\Phi$ is a functional on convex bodies and $K_t$ is a perturbation of $K$ involving a bounded function $g$. In the case considered above, $\Phi$ is the volume, $g=h_L$, and $K_t$ is the Wulff shape of $h_K+th_L$ (for sufficiently small $|t|$). 

The literature knows several examples of measures $\mu$ defined by convex bodies that arise in the form (\ref{1.3}) and/or (\ref{1.4}). Besides the surface area measure, we mention only the cone-volume measure, the dual curvature measures, the surface area measures in the $L_p$ or more generally in the Orlicz Brunn--Minkowski theory, and the chord measures. For more information, in particular about the corresponding Minkowski problems, we refer to the recent survey article by Huang, Yang and Zhang \cite{HYZ25}.

The last decades have seen a very successful generalization of various results on convex bodies to results on convex functions or log-concave functions. For a survey, we refer to \cite[Chap. 9]{AGM21}. A function $f:\R^n\to \R$ is log-concave if $f=e^{-\varphi}$ with a convex function $\varphi:\R^n\to(-\infty,\infty]$. If $K\subset\R^n$ is a convex body and
$$ I_K^\infty(x) :=\left\{\begin{array}{ll} 0 & \mbox{if } x\in K,\\ \infty & \mbox{if }x\in\R^n\setminus K, \end{array}\right.\qquad 
{\mathbbm 1}_K(x):=\left\{\begin{array}{ll} 1 & \mbox{if } x\in K,\\ 0 & \mbox{if }x\in\R^n\setminus K \end{array}\right.$$
defines its indicator function $I_K^\infty$ and characteristic function ${\mathbbm 1}_K$, then $I_K^\infty$ is convex and ${\mathbbm 1}_K$ is log-concave. Moreover,
$I_K^\infty= h_K^*$, where $h_K$ denotes the support function of $K$ and 
$f^*$ is the Legendre transform of $f$. As 
$$(I_K^\infty)^*+(I_L^\infty)^*= h_K^{**}+h_L^{**}=h_K+h_L=h_{K+L}=h_{K+L}^{**}= (I^\infty_{K+L})^*$$
and the Minkowski sum is a central notion in the theory of convex bodies, this indicates already why the Legendre transform plays an essential role when results on convex bodies are carried over to convex functions.

For a log-concave function $f=e^{-\varphi}$, it is natural to consider the measure $\nu_f$ on $\R^n$ which has density $e^{-\varphi}$ with respect to Lebesgue measure, and to view its total measure 
$$ J(f)= \int_{\R^n} e^{-\varphi(x)}\,\D x= \int_{\R^n}f(x)\,\D x$$
as the `total mass' of $f$, assuming that it is finite. We also consider its image measure
$$ S_f:= (\nabla\varphi)_\sharp\nu_f,$$
where $\nabla \varphi$ is the (almost everywhere on ${\rm dom}\,\varphi$ defined) gradient of the convex function $\varphi$. Defining
$$ f_t:= e^{-(\varphi^*+tg)^*}$$
(so that $f_0=f$) with a function $g$, we have in analogy to results on convex bodies, under suitable assumptions, that
\begin{equation}\label{1.6}
\frac{\D}{\D t} J(f_t)\Big|_{t=0} =\int_{\R^n} g\,\D S_f.
\end{equation}
We refer to Proposition 5.4 in Rotem \cite{Rot22}, where the assumptions are that $0<\int f<\infty$ and $g:\R^n\to\R$ is continuous and bounded. There are several variants with different assumptions; see Colesanti and Fragal\`a \cite{CF13}, Cordero--Erausquin and Klartag \cite{CK15}, Ulivelli \cite{Uli24}. In \cite[Def. 1]{CK15}, the measure $S_f$ is called a moment measure, a terminology which can only be understood from the references in \cite{CK15}. Rotem \cite{Rot22} calls this measure the surface area measure of the log-concave function $f$.

We mention that Santambrogio \cite{San16} has reproved results of \cite{CK15} by measure transport methods. For a result similar to (\ref{1.6}) we also refer to the case $q=n$ of Lemma 5.2 in \cite{HLXZ16}.

\section{An analogue for pseudo-cones}\label{sec2}

The purpose of the following is to find measures analogous to $S_f$ for pseudo-cones instead of convex functions. We need a few explanations before we can formulate our problem. We work in Euclidean space $\R^n$ ($n\ge 2$), with scalar product $\langle\cdot\,,\cdot\rangle$, origin $o$, and unit sphere $\Sn$. A  closed convex cone $C\subset\R^n$ is given, with nonempty interior and not containing a line. Its dual cone $C^\circ:= \{x\in\R^n: \langle x,y\rangle\le 0\;\forall y\in C\}$ then has the same properties. We write $\Omega_C:= \Sn\cap{\rm int}\,C$ and $\Omega_{C^\circ}:= \Sn\cap {\rm int}\, C^\circ$. We denote by ${\rm int}$ the interior and by  ${\rm cl}$ the closure, in $\Sn$ as well as in $\R^n$.

A $C$-pseudo-cone is a closed convex set $K\subset C$ with $o\notin K$ for which $C$ is the recession cone (that is, $C={\rm rec}\,K:=\{z\in \R^n: K+z \subseteq K\}$). This implies that $K+C= K$. The set of all $C$-pseudo-cones is denoted by $ps(C)$. This set can be considered, in several respects, as a counterpart to the set of convex bodies containing $o$ in the interior. For example, surface area measures and their analogues and generalizations can be defined and corresponding Minkowski type problems can be investigated, the copolarity of pseudo-cones has many properties similar to those of the polarity of convex bodies, and the Gauss image problem can be treated for pseudo-cones (see, e.g., \cite{Sch18} -- \cite{Sch25d}). The gist of the following is that $C$-pseudo-cones, or rather their radial functions, can also be considered as a counterpart to the set of convex functions with bounded domain.

We need a few more notions to make this precise. Let $K\in ps(C)$. The support function of $K$ is defined by
$$ h_K(x):= \sup\{\langle x,y\rangle: y\in K\}\quad\mbox{for }x\in C^\circ.$$
The supremum is a maximum if $u\in {\rm int}\,C^\circ$. Since $h_K\le 0$, we write $\overline h_K:=-h_K$. 

The {\em radial function} of $K$ is defined by
$$ \rho_K(v):= \min\{r>0: rv\in K\}\quad\mbox{for }v\in\Omega_C.$$
Note that the minimum exists, since $C+x\subseteq K$ for any $x\in K$ and hence $\{r>0: rv \in K\}\not=\emptyset$ for $v\in\Omega_C$, and $K$ is closed. 

Let $\eta_K$ be the set of all $v\in\Omega_C$ for which the outer unit normal vector of $K$ at the boundary point $\rho_K(v)v$ is not unique. Then, as is well known, $\eta_K$ has spherical Lebesgue measure zero. The {\em radial Gauss map} $\alpha_K: \Omega_C\setminus\eta_K\to {\rm cl}\,\Omega_{C^\circ}$ is defined by letting $\alpha_K(v)$ be the unique outer unit normal vector of $K$ at $\rho_K(v)v$, for $v\in \Omega_C\setminus\eta_K$. 

The setting described in Section \ref{sec1} is briefly the following. We choose the function $F(x)=e^{-x}$; then for each convex function $\varphi$ on $\R^n$ the total mass of $f:= F\circ\varphi$ is defined by $J(f):= \int_{\R^n} f(x)\,\D x$. Let $\nu_f$ be the measure with density $F\circ \varphi$ with respect to Lebesgue measure, let $S_f:= (\nabla\varphi)_\sharp \nu_f$ be the push-forward of $\nu_f$ under the gradient map of $\varphi$. If $f_t$ denotes a suitable perturbation of $f$, so that $f_0=f$, involving a function $g$, then, under suitable assumptions, $\frac{\D}{\D t}J(f_t)\big|_{t=0} = \int_{\R^n}g\,\D S_f$.

We want to imitate this for pseudo-cones, replacing the convex function $\varphi$ by the radial function of a pseudo-cone and the function $x\mapsto e^{-x}$ by a rather general function. For this, we take a continuously differentiable, strictly decreasing function $G:(0,\infty)\to (0,\infty)$ and define $F(x) := -xG'(x)$ for $x>0$. We 
assume that for every $a>0$, the function $F|_{[a,\infty)}$ is bounded. Typical examples are $G(x)=e^{-x}$ and $F(x)=xe^{-x}$, or $G(x)= x^q$ with $q<0$ and $F(x)=|q| x^q$. For $K\in ps(C)$, let $\nu_{F,K}$ be the measure on $\Omega_C$ with density $F\circ\rho_K$ with respect to spherical Lebesgue measure. This measure is finite, since $\rho_K$ is bounded away from zero, so that the function $F\circ\rho_K$ is bounded. Instead of the gradient map of a convex function, we now employ the radial Gauss map of $K$. Thus we denote by 
$$\mu_{F,K}:= \alpha_K{_\sharp} \nu_{F,K}$$
the finite measure on ${\rm cl}\,\Omega_{C^\circ}$ which is the push-forward of $\nu_{F,K}$ under $\alpha_K$. We call this measure the {\em $F$-radial measure} of $K$. 

The following should be observed. If $z\in C\setminus\{o\}$, then $K:= C+z\in ps(C)$. In this case, $\alpha_K$ maps almost all of $\Omega_C$ to the boundary of $\Omega_{C^\circ}$, so that $\mu_{F,K}(\Omega_{C^\circ})=0$. This shows that for radial measures of $C$-pseudo-cones, the natural domain is ${\rm cl}\,\Omega_{C^\circ}$, and not $\Omega_{C^\circ}$.

Our first aim is to show that the measure $\mu_{F,K}$ can also be interpreted as a derivative measure, in analogy to (\ref{1.4}) and (\ref{1.6}). For this, we define the functional $J_G$ by
$$ J_G(K):= \int_{\Omega_C} (G\circ \rho_K)(v)\,\D v\quad\mbox{for } K\in ps(C),$$
where $\D v$ indicates integration with respect to the spherical Lebesgue measure. By our assumptions on $G$, the functional $J_G$ is finite. 

The Wulff shape appearing below will be defined in Section \ref{sec3}.

\begin{theorem}\label{T2.1}
Let $F$ and $G$ be as defined above. Let $K\in ps(C)$. Let $g:{\rm cl}\,\Omega_{C^\circ}\to\R$ be continuous. Let $K_t$ be the Wulff shape associated with $(C, \overline h_Ke^{tg})$ for $|t|\le 1$ (say). Then
$$ \frac{\D}{\D t} J_G(K_t)\Big|_{t=0} = -\int_{{\rm cl}\,\Omega_{C^\circ}} g\,\D \mu_{F,K}.$$
\end{theorem}

We prove this in Section \ref{sec3}, after some preliminaries.

Our second aim would be to solve a Minkowski type existence problem for the measure $\mu_{F,K}$: what are necessary and sufficient conditions on a Borel measure $\mu$ on ${\rm cl}\,\Omega_{C^\circ}$ in order that there is a pseudo-cone $K\in ps(C)$ for which $\mu=\mu_{F,K}$? We give only an answer for the special functions $G(x)= x^q$, $F(x)=|q|x^q$ with $q<0$. In this case, we have
\begin{equation}\label{2.1} 
\mu_{F,K}= |q|n\widetilde C_q(K,\cdot),\qquad J_G(K)= n \widetilde V_q(K).
\end{equation}
Here $\widetilde V_q(K)$ is the $q$th dual volume of $K$ and $\widetilde C_q(K,\cdot)$ denotes the $q$th dual curvature measure of $K$ (see, e.g., \cite{LYZ24}, formula (4.9) and Def. 4.3). In this way, we give a solution to a problem left open in \cite{LYZ24}, as explained in Section \ref{sec4}.

\section{Proof of Theorem \ref{T2.1}}\label{sec3}

We begin with some preliminaries. Let $K\in ps(C)$. We have already defined its support function on $C^\circ$. Being a convex function, it is continuous on the interior of its domain and hence on ${\rm int}\,C^\circ$. We show that $h_K$ is continuous on $C^\circ$. Let $x_j\in C^\circ$ for $j\in\N_0$ and suppose that $x_j\to x_0$ as $j\to\infty$. Since $h_K$ is positively homogeneous of degree one, we can assume that $x_0\not= o$. The bounded sequence $(h_K(x_j))_{j\in\N}$ has a convergent subsequence. If there were convergent subsequences with different limits, then $K$ would have two distinct supporting hyperplanes (in the sense of \cite[p. 45]{Sch14}) with normal vector $x_0$, a contradiction. It follows that $h_K(x_j)\to h_K(x_0)$, which gives the assertion. In the following, we need only $\overline h_K$ on ${\rm cl }\,\Omega_{C^\circ}$.

Let $K\in ps(C)$ and $v\in \Omega_C$. Then $\rho_K(v)v$ is in the boundary of $K$. Let $u_*$ be an outer unit normal vector of $K$ at this point. Then
\begin{equation}\label{3.1}
\overline h_K(u_*) = |\langle \rho_K(v)v, u_*\rangle|
\end{equation}
and
\begin{equation}\label{3.2}
\overline h_K(u) \le |\langle \rho_K(v)v, u\rangle|\quad\mbox{for all }u\in{\rm cl}\,\Omega_{C^\circ}.
\end{equation}

First we modify the definition of the Wulff shape. Let $f:{\rm cl}\,\Omega_{C^\circ}\to [0,\infty)$ be a continuous function which is not identically zero. We define
$$ [f]:= C\cap\bigcap_{u\in{\rm cl}\,\Omega_{C^\circ}} \{y\in\R^n: |\langle y,u\rangle| \ge f(u)\}$$
and call this the {\em Wulff shape} associated with $(C,f)$. It is nonempty, since $f$ is bounded, and it does not contain $o$ since $f$ is positive somewhere on ${\rm cl}\,\Omega_{C^\circ}$. Being an intersection of $C$-pseudo-cones, it is clearly a $C$-pseudo-cone. It follows from the definition that
\begin{equation}\label{3.2a} 
[\overline h_K]= K\quad\mbox{for }K\in ps(C)
\end{equation}
and
\begin{equation}\label{nV2}
f(u)\le \overline h_{[f]}(u) \quad\mbox{for }u\in{\rm cl}\,\Omega_{C^\circ}.
\end{equation}

Let $v\in \Omega_C$. Then $\rho_{[f]}(v)v$ is a boundary point of $[f]$. By (\ref{3.2}) and (\ref{nV2}) we get
\begin{equation}\label{3.3}
f(u)\le |\langle \rho_{[f]}(v)v,u\rangle|\quad \mbox{for all } u\in{\rm cl}\,\Omega_{C^\circ}.
\end{equation}
Let $u_*$ be an outer unit normal vector of $[f]$ at $\rho_{[f]}(v)v$. By (\ref{3.1}) we have $\overline h_{[f]}(u_*) = |\langle \rho_{[f]}(v)v,u_*\rangle|$. By (\ref{nV2}), $\overline h_{[f]}(u_*) = f(u_*)+\varepsilon$ with $\varepsilon\ge 0$. This implies that
$$ [f]\subset \{y\in\R^n: |\langle y,u_*\rangle| \ge f(u_*)+\varepsilon\},$$
so that from $\varepsilon>0$ it would follow that $\rho_K(v)v\notin [f]$, a contradiction. From $\varepsilon =0$ we see that 
\begin{equation}\label{3.4}
f(u_*) = |\langle \rho_{[f]}(v)v,u_*\rangle|.
\end{equation}

We can write (\ref{3.3}) and (\ref{3.4}) also in the form
\begin{equation}\label{3.3a}
\rho_{[f]}(v) \ge \frac{f(u)}{|\langle v,u\rangle|} \quad \mbox{for }u\in{\rm cl}\,\Omega_{C^\circ}
\end{equation}
and
\begin{equation}\label{3.3b}
\rho_{[f]}(v) =\frac{f(u_*)}{|\langle v,u_*\rangle|}.
\end{equation}
(Note that $\langle v,u\rangle\not=0$ for $v\in\Omega_C$ and $u\in{\rm cl}\,\Omega_{C^\circ}$.) Thus, the radial function of the Wulff shape is given by
$$ \rho_{[f]}(v) = \max_{u\in{\rm cl}\,\Omega_{C^\circ}}\frac{f(u)}{|\langle v,u\rangle|}$$
for $v\in\Omega_C$, and the maximum is attained at $u_*$ if and only if $u_*$ is an outer unit normal vector of $[f]$ at $\rho_{[f]}(v)v$.

We need a different form of Lemma 9 in \cite{Sch24b}. More precisely, in \cite{Sch24b} the lemma involved a compact subset $\omega$ of $\Omega_{C^\circ}$, because it was needed in that form. Now this set is replaced by ${\rm cl}\,\Omega_{C^\circ}$ (thus, the assumption that $K\in\K(C,\omega)$ is no longer necessary). Since there are other changes, too, and parts of the proof are needed later, we give the full proof. We emphasize, however, as in \cite{Sch24b}, that the main ideas (though changed) are taken from \cite{HLYZ16}.

\begin{lemma}\label{L3.1}
Let $K\in ps(C)$. Let $g:{\rm cl}\,\Omega_{C^\circ}\to \R$ be continuous. Define the function $h_t$ (for $|t|\le 1$, say) by
$$ h_t(u) := \overline h_K(u)e^{tg(u)}\quad\mbox{for }u\in{\rm cl}\,\Omega_{C^\circ},$$
and let $[h_t]$ be the Wulff shape associated with $(C,h_t)$. Then, for $v\in\Omega_C\setminus \eta_K$,
$$ \lim_{t\to 0} \frac{\log \rho_{[h_t]}(v) - \log\rho_K(v)}{t}= g(\alpha_K(v)).$$
\end{lemma}

\begin{proof}
Let $v\in \Omega_C\setminus \eta_K$. For $|t|\le 1$, let $u_t$ be an outer unit normal vector of the pseudo-cone $[h_t]$ at its boundary point $\rho_{[h_t]}(v)v$, so that by (\ref{3.4}) and (\ref{3.3}) we have
\begin{equation}\label{3.5}
h_t(u_t) = |\langle \rho_{[h_t]}(v)v,u_t\rangle| \quad\mbox{and}\quad h_t(u) \le |\langle \rho_{[h_t]}(v)v,u\rangle|\quad\mbox{for all }u\in{\rm cl}\,\Omega_{C^\circ}.
\end{equation}
We prove first that
\begin{equation}\label{3.6}
\lim_{t\to 0} u_t=u_0.
\end{equation}
Since ${\rm cl}\,\Omega_{C^\circ}$ is compact, there is a sequence $(u_{t_k})_{k\in\N}$ with $t_k\to 0$ that converges to some $u'\in{\rm cl}\,\Omega_{C^\circ}$. Using the uniform convergence $h_t\to h_0=\overline h_K$ on ${\rm cl}\,\Omega_{C^\circ}$ (since $|g|$ is bounded) and the convergence $\rho_{[h_t]}\to\rho_{[h_0]}= \rho_K$ (which follows from Lemma 5 in \cite{Sch18}, with $\omega$ replaced by ${\rm cl}\,\Omega_{C^\circ}$), we see from (\ref{3.5}), which implies
$$ h_{t_k}(u_{t_k}) = |\langle \rho_{[h_{t_k}]}(v)v,u_{t_k}\rangle|\quad\mbox{for }k\in\N,$$
that
$$ \overline h_K(u') =|\langle\rho_K(v)v,u'\rangle|.$$
Since
$$ \overline h_K(u) \le |\langle \rho_K(v)v,u\rangle|\quad \mbox{for all } u\in{\rm cl}\,\Omega_{C^\circ}$$
by (\ref{3.2}), this means that $u'$ is an outer unit normal vector of $K$ at $\rho_K(v)v$. Therefore (and since $v\notin\eta_K$), $u'= u_0$. Thus, every convergent sequence $(u_{t_k})_{k\in\N}$ with $t_k\to 0$ has the limit $u_0$. This proves (\ref{3.6}).

Using (\ref{3.3b}) for $f=h_t$, $u_*=u_t$ and then (\ref{3.2}) for $u=u_t$, we get
\begin{eqnarray*}
&& \log \rho_{[h_t]}(v)-\log \rho_K(v) =\log h_t(u_t) -\log|\langle v,u_t\rangle|-\log \rho_K(v)\\
&&\le \log h_t(u_t)-\log\overline h_K(u_t)= tg(u_t).
\end{eqnarray*}
Using (\ref{3.3b}) for $f=h_t$, $t=0$ and then (\ref{3.3a}) for $f=h_t$, $u=u_0$, we get
\begin{eqnarray*}
&& \log \rho_{[h_t]}(v)-\log \rho_K(v)= \log \rho_{[h_t]}(v)-\log\overline h_K(u_0) +\log|\langle v,u_0\rangle|\\
&&\ge \log h_t(u_0) -\log \overline h_K(u_0)= tg(u_0).
\end{eqnarray*}
It follows that
\begin{equation}\label{3.7}
\Bigg|\frac{\log \rho_{[h_t]}(v)-\log \rho_K(v)}{t}-g(u_0)\Bigg|\le|g(u_t)-g(u_0)|,
\end{equation}
so that (\ref{3.6}) and the continuity of $g$ give the assertion.
\end{proof}

We can now prove Theorem \ref{T2.1}.

\begin{proof}
For a positive differentiable function $\psi$ we have $\psi'=\frac{\psi'}{\psi}\psi= (\log \psi)'\psi$. Using this, we get under the assumptions of Theorem \ref{T2.1} that, for $v\in\Omega_C\setminus \eta_K$,
\begin{eqnarray*}
\frac{\D}{\D t}(G\circ \rho_{K_t})(v)\Big|_{t=0} &=& G'(\rho_K(v))\frac{\D}{\D t} \rho_{K_t}(v)\Big|_{t=0}\\
&=& G'(\rho_K(v))\rho_K(v)\frac{\D}{\D t} \log \rho_{K_t}(v)\Big|_{t=0}\\
&=& -(F\circ \rho_K)(v)g(\alpha_K(v)),
\end{eqnarray*}
where $G'(x)x=-F(x)$ and Lemma \ref{L3.1} were applied. 

By (\ref{3.7}) and the boundedness of $g$, there are constants $\delta>0$ and $b>0$ such that
$$ |\log\rho_{K_t}(v) -\log\rho_K(v)|\le b|t|$$
for $|t|<\delta$ and almost all $v\in\Omega_C$. By the continuity of the radial function, this holds for all $v\in\Omega_C$. We can assume that $\delta$ and the constant $a>0$ are chosen so that $\rho_{K_t}\ge a$ for $|t|<\delta$. The function $\Psi$ defined by $\Psi(x)= G(e^x)$ for $x\ge \log a$ is positive and has a negative derivative, hence it has a Lipschitz constant $L$. It follows that, for $v\in\Omega_C$, 
\begin{eqnarray*}
|(G\circ \rho_{K_t})(v) - (G\circ \rho_K)(v)| &=& |\Psi(\log\rho_{K_t}(v)) - \Psi(\log\rho_K(v))|\\
&\le& L|\log\rho_{K_t}(v) -\log\rho_K(v)|\le Lb|t|.
\end{eqnarray*}
Now we can apply the dominated convergence theorem and obtain
\begin{eqnarray*}
\frac{\D}{\D t} J_G(K_t)\Big|_{t=0} &=& \int_{\Omega_C} \frac{\D}{\D t}(G\circ \rho_{[K_t]})(v)\Big|_{t=0}\,\D v= -\int_{\Omega_C} g(\alpha_K(v)) (F\circ\rho_K)(v)\,\D v\\
&=& -\int_{\Omega_C} g(\alpha_K(v))\,\nu_{F,K}(\D v) = -\int_{{\rm cl}\,\Omega_{C^\circ}} g\,\D\mu_{F,K},
\end{eqnarray*}
where the general transformation theorem for integrals (or change of variable formula) was used. This completes the proof of Theorem \ref{T2.1}.
\end{proof}

\section{The dual Minkowski problem for negative exponents}\label{sec4}

We now choose $G(x)=x^q$ and  $F(x)=|q|x^q$ with $q<0$. According to (\ref{2.1}), for $K\in ps(C)$ this gives
\begin{equation}\label{4.0} 
\frac{1}{|q|n}\mu_{F,K}(\omega)= \frac{1}{n} \int_{\alpha_K^{-1}(\omega)} \rho_K(v)^q\,\D v = \widetilde C_q(K,\omega)
\end{equation}
for Borel sets $\omega\subset {\rm cl} \,\Omega_{C^\circ}$ and
\begin{equation}\label{4.0a}
\frac{1}{n}J_G(K) = \frac{1}{n} \int_{\Omega_C} \rho_K(v)^q\,\D v=\widetilde V_q(K).
\end{equation}
The measure $\widetilde C_q(K,\cdot)$ is known as the $q$th dual curvature measure of $K$ (note that $\alpha_K^{-1}(\omega)= {\boldsymbol\alpha}_K^*(\omega)\setminus\eta_K$ and $\Ha^{n-1}(\eta_K)=0$; ${\boldsymbol\alpha}_K^*$ is defined in \cite[(4.6)]{LYZ24}). Originally, these measures were introduced for convex bodies, see \cite{HLYZ16}.  Whereas a `duality' or analogy to the classical curvature measures can only be observed for integers $q\in\{0,\dots,n\}$, the measures $\widetilde C_q(K,\cdot)$ can be defined for all $q\in\R$. For pseudo-cones, these measures were introduced in \cite{LYZ24}. 

It should be observed  that $\widetilde C_q(K,\cdot)$ for $K\in ps(C)$ is a measure on ${\rm cl}\,\Omega_{C^\circ}$ and not necessarily on $\Omega_{C^\circ}$. For example, if $K=C+z$ with $z\in{\rm int}\,C$, then $K\in ps(C)$ and $\widetilde C_q(K,\cdot)$ is concentrated on ${\rm cl}\,\Omega_{C^\circ}\setminus \Omega_{C^\circ}$. For this reason, in the formula
$$ \int_{{\rm cl}\,\Omega_{C^\circ}} g(u)\,\widetilde C_q(K,\D u) = \frac{1}{n} \int_{\Omega_C} g(\alpha_K(v))\rho_K(v)^q\,\D v$$
for bounded measurable functions $g:{\rm cl}\,\Omega_{C^\circ}\to \R$, which follows from the change of variable formula, it is essential that the left-hand integral extends over the closure of $\Omega_{C\circ}$ and not only over $\Omega_{C\circ}$ (otherwise, the left integral could be zero, whereas the right integral is not).

In Li, Ye and Zhu \cite{LYZ24}, the following is proved. If $\varphi$ is a nonzero finite Borel measure on $\Omega_{C^\circ}$ whose support is a compact subset of $\Omega_{C^\circ}$ and if $q\not= 0$, then there exists $K\in ps(C)$ with $\widetilde C_q(K,\cdot)= \varphi$ (Theorem 6.1). If $\varphi$ is a nonzero finite Borel measure on $\Omega_{C^\circ}$ and if $q>0$, then there exists $K\in ps(C)$ with $\widetilde C_q(K,\cdot)= \varphi$ (Theorem 7.1). Thus, the authors solve the Minkowski existence problem for the $q$th dual curvature measure of pseudo-cones if either $q\not=0$ and the given measure has compact support, or if $q>0$, but only for finite measures (which is a severe restriction if $q>0$). They express (on page 2006) the opinion that the solvability for $q<0$, when $\varphi$ does not concentrate on a compact subset of $\Omega_{C^\circ}$, `seems to be much more intractable'. However, we prove the following.

\begin{theorem}\label{T4.1}
Let $q<0$. If $\varphi$ is a nonzero finite Borel measure on ${\rm cl}\,\Omega_{C^\circ}$, then there exists a pseudo-cone $K\in ps(C)$ with $\widetilde C_q(K,\cdot)=\varphi$.
\end{theorem}

The proof is prepared by the following lemma. For $K\in ps(C)$, we denote by $b(K)$ the distance of $K$ from the origin.

\begin{lemma}\label{L4.1}
Let $q<0$. There are positive constants $c_1,c_2$, depending only on $C,q,n$, with the following property. If $K\in ps(C)$ and $\widetilde V_q(K)=1$, then $c_1\le b(K)\le c_2$.
\end{lemma}

\begin{proof}
We choose a number $a$ with
\begin{equation}\label{4.1}
\frac{1}{n}a^q\sigma_{n-1}(\Omega_C)>1,
\end{equation}
where $\sigma_{n-1}$ denotes the spherical Lebesgue measure on $\Sn$. For $x\in C$ and $\varepsilon>0$, let 
$$ A(x):=\{v\in\Omega_C:\rho_{C+x}(v)\le a\}$$ 
(which may be empty if $|x|$ is too large). If $K\in ps(C)$ and $x\in C\setminus\{o\}$, we have $C+x\subseteq K$, thus
$$ \rho_K(v)\le \rho_{C+x}(v)\le a\quad\mbox{for } v\in A(x)$$
and hence (since $q<0$)
$$ \widetilde V_q(K) \ge \frac{1}{n} \int_{A(x)} \rho_K(v)^q\,\D v\ge\frac{1}{n} a^q\sigma_{n-1}(A(x)).$$
Let $K\in ps(C)$ be a pseudo-cone with $b(K)=\varepsilon$. Then there is a point $x_\varepsilon\in K$ with $|x_\varepsilon|=\varepsilon$. Since $\sigma_{n-1}(A(x))\to \sigma_{n-1}(\Omega_C)$ as $x\to o$, we get
$$ \widetilde V_q(K)\ge \frac{1}{n}a^q\sigma_{n-1}(A(x_\varepsilon))\to \frac{1}{n}a^q\sigma_{n-1}(\Omega_C) \quad\mbox{as } \varepsilon \to 0.$$
This contradicts (\ref{4.1}), if there are $C$-pseudo-cones $K$ satisfying $\widetilde V_q(K)=1$ with arbitrarily small distance from the origin. The existence of the constant $c_1$ follows.

To prove the existence of $c_2$, let $K\in ps(C)$ and $\widetilde V_q(K)=1$. Define
$$ r_0:= \max\{r>0: rB^n\cap{\rm int}\,K=\emptyset\},$$
where $B^n$ denotes the unit ball in $\R^n$, centered at $o$. Then $\rho_K(v)\ge r_0$ for all $v\in\Omega_C$ and hence
$$ 1= \widetilde V_q(K) \le \frac{1}{n}\int_{\Omega_C} r_0^q\,\D v =\frac{1}{n} r_0^q\sigma_{n-1}(\Omega_C),$$
thus $r_0\le c_2$ with a constant $c_2$ depending only on $C,q,n$. By the definition of $r_0$, there is a vector $x\in K$ with $|x|=r_0\le c_2$. This implies $b(K)\le c_2$.
\end{proof}

Now we can prove Theorem \ref{T4.1}.

By $C^{\ge}({\rm cl}\,\Omega_{C^\circ})$ we denote the space of nonnegative, continuous functions on ${\rm cl}\,\Omega_{C^\circ}$ which are not identically zero. For $f\in C^{\ge}({\rm cl}\,\Omega_{C^\circ})$, the Wulff shape $[f]$ is a well-defined $C$-pseudo-cone.

Let $\varphi$ be a nonzero, finite Borel measure on ${\rm cl}\,\Omega_{C^\circ}$. Its total measure is denoted by $|\varphi|$. Let $q<0$.

For $f\in C^{\ge}({\rm cl}\,\Omega_{C^\circ})$ we define (as in \cite{LYZ24}, but with modifications)
$$ \Phi(f) := -\frac{1}{|\varphi|} \int_{{\rm cl}\,\Omega_{C^\circ}} \log f(u)\,\varphi(\D u) +\frac{1}{q}\log \widetilde V_q([f]).$$
We may have $\Phi(f)=\infty$ now, since $f$ is allowed to be zero, but $\Phi(f)$ is finite for some $f$, in particular for $\overline h_L$ for suitable $L\in ps(C)$. It is easily checked that $\Phi(\lambda f)= \Phi(f)$ for $\lambda>0$.

We can argue similarly as in \cite[Sect. 9]{Sch18}. We state  first that $\Phi$ attains a minimum on the set 
$${\mathcal L}':= \{\overline h_L: L\in ps(C),\,\widetilde V_q(L)=1\}.$$ 
For the proof, choose $L\in ps(C)$ with $\widetilde V_q(L)=1$. By the proof of Lemma \ref{L4.1}, there is a constant $c$, independent of $L$, such that $\overline h_L\le c$. It follows that $\Phi(\overline h_L)\ge -\log c$ and hence $\inf \{\Phi(f):f\in{\mathcal L}'\}>-\infty$. Also, there are $L\in{\mathcal L}'$ with $\Phi(\overline h_L)<\infty$.

Let $(K_i)_{i\in\N}$ be a sequence with $\overline h_{K_i}\in{\mathcal L}'$ such that
$$ \lim_{i\to\infty} \Phi(\overline h_{K_i}) = \inf\{\Phi(f):f\in{\mathcal L}'\}.$$
By Lemma 1 in \cite{Sch24a} and Lemma \ref{L4.1}, the sequence $(K_i)_{i\in\N}$ has a subsequence converging to a pseudo-cone $K_0\in ps(C)$. The corresponding support functions converge pointwise. By Lemma \ref{L4.1}, they are uniformly bounded. Using the dominated convergence theorem, we see that the minimum of $\Phi$ on ${\mathcal L}'$ is attained at $\overline h_{K_0}$. Modifying the notation from \cite[Sect. 6]{LYZ24}, we write
\begin{eqnarray*}
\beta_{\rm fct} &:=& \inf\{\Phi(f): f\in C^{\ge}({\rm cl}\,\Omega_{C^\circ})\},\\
\beta_{\rm pc} &:=& \inf\{\Phi(\overline h_K): K\in ps(C)\}.
\end{eqnarray*}
Since $\Phi$ is homogeneous of degree zero and $\widetilde V_q$ is homogeneous of degree $q\not= 0$, we have shown that $\beta_{\rm pc}=\Phi(\overline h_{K_0})$. Trivially, $\beta_{\rm pc} \ge \beta_{\rm fct}$. By (\ref{nV2}) we have $f(u)\le \overline h_{[f]}(u)$ for $u\in{\rm cl}\,\Omega_{C^\circ}$ and hence (observing that $[\overline h_{[f]}]=[f]$ by (\ref{3.2a})) $\Phi(f)\ge \Phi(\overline h_{[f]})$, thus $\beta_{\rm fct}\ge \beta_{\rm pc}$. We have obtained that $\beta_{\rm fct}= \Phi(\overline h_{K_0})$.

Now let $g$ be a continuous function on ${\rm cl}\,\Omega_{C^\circ}$ and let $K_t=[h_t]$ be the Wulff shape associated with $(C,h_t)$ (for sufficiently small $|t|>0$), where $h_t= \overline h_{K_0}e^{tg}$. Then
$$ \Phi(h_t) = -\frac{1}{|\varphi|} \int_{{\rm cl}\,\Omega_{C^\circ}} \log h_t(u)\,\varphi(\D u) +\frac{1}{q} \log\widetilde V_q([h_t]).$$
By Theorem \ref{T2.1}, together with (\ref{4.0}) and (\ref{4.0a}), we have (since $q<0$)
$$ \frac{\D}{\D t} \widetilde V_q([h_t])\Big|_{t=0} = q\int_{{\rm cl}\,\Omega_{C^\circ}} g(u)\,\widetilde C_q(K_0,\D u).$$

We know that $\beta_{\rm fct} = \Phi(\overline h_{K_0})$, the function $t\mapsto \Phi(h_t)$ is finite and differentiable at $0$, hence $\frac{\D \Phi(h_t)}{\D t}\big|_{t=0}=0$, and $[h_0]=[\overline h_{K_0}]= K_0$ by (\ref{3.2a}). Therefore,
$$ 0= -\frac{1}{|\varphi|} \int_{{\rm cl}\,\Omega_{C^\circ}} g(u)\,\varphi(\D u) +\frac{1}{\widetilde V_q(K_0)} \int_{{\rm cl}\,\Omega_{C^\circ}} g(u)\,\widetilde C_q(K_0,\D u).$$
Since this holds for all continuous functions $g$ on  ${\rm cl}\,\Omega_{C^\circ}$, we deduce that
$$ \varphi = \frac{|\varphi|}{\widetilde V_q(K_0)}\widetilde C_q(K_0,\cdot).$$
Since $\widetilde C_q(\lambda K_0,\cdot) = \lambda^q\widetilde C_q(K_0,\cdot)$ for $\lambda>0$, we get $\varphi= \widetilde C_q(\lambda K_0,\cdot)$ for suitable $\lambda>0$, which completes the proof of Theorem \ref{T4.1}.

The existence result proved by Theorem \ref{T4.1} raises a uniqueness question, but this remains open. For convex bodies, a uniqueness proof was provided by Zhao \cite{Zha17}, but it is at the moment not clear how it could be carried over to pseudo-cones.

\noindent Author's address:\\[2mm]
Rolf Schneider\\Mathematisches Institut, Albert--Ludwigs-Universit{\"a}t\\D-79104 Freiburg i.~Br., Germany\\E-mail: rolf.schneider@math.uni-freiburg.de

\end{document}